\documentclass[11pt,reqno, english]{amsart}
\usepackage[a4paper]{geometry}
\usepackage{amsmath,amssymb,amscd,amsthm,amsfonts}
\usepackage{graphicx,subcaption}
\usepackage{hyperref}
\usepackage{dsfont}
\usepackage{xcolor}
\usepackage{color,soul}
\usepackage[utf8]{inputenc}
\usepackage[nobysame, alphabetic]{amsrefs}
\usepackage{tikz}
\usepackage[capitalise]{cleveref}

\usetikzlibrary{arrows.meta}
\newcommand{\PS}{0.0042}
\newcommand{\xyz}[4]{%
  \pgfmathsetmacro{\tmpx}{\PS*(((-(#2))+(#3))/1.41421356)}%
  \pgfmathsetmacro{\tmpy}{\PS*(((-(#2))-(#3)+2*(#4))/2.44948975)}%
  \coordinate (#1) at (\tmpx,\tmpy);%
}
\definecolor{nA}{rgb}{0.86,0.08,0.10}
\definecolor{nB}{rgb}{0.00,0.62,0.62}
\definecolor{nC}{rgb}{0.15,0.62,0.18}
\definecolor{nD}{rgb}{0.78,0.10,0.74}
\definecolor{nE}{rgb}{0.12,0.30,0.85}
\definecolor{nF}{rgb}{0.95,0.55,0.00}

\newtheorem{theorem}{Theorem}[section]
\newtheorem{corollary}[theorem]{Corollary}
\newtheorem{lemma}[theorem]{Lemma}

\newtheorem{conjecture}[theorem]{Conjecture}
\newtheorem{question}[theorem]{Question}

\theoremstyle{definition}

\newtheorem{claim}[theorem]{Claim}

\newcommand{\rr}{\mathds{R}}

\DeclareMathOperator{\conv}{conv}

\DeclareMathOperator{\supp}{supp}

\title[-]{Bisecting masses with families of parallel hyperplanes}

\author[Hubard]{Alfredo Hubard}\address{University Gustave Eiffel Paris-Est} 
\email{alfredo.hubard@univ-eiffel.fr}

\author[Sober\'on]{Pablo Sober\'on}\address{Baruch College \& The Graduate Center, City University of New York, New York, NY 10010} 
\email{psoberon@gc.cuny.edu}



\thanks{The research of P. Sober\'on was supported by NSF CAREER award no. 2237324, NSF award no. 2054419 and a PSC-CUNY Trad B award.}

\begin{document}

\begin{abstract}
We prove a common generalization of several mass partition results using hyperplane arrangements to split $\rr^d$ into two sets.  Our main result implies the ham sandwich theorem, the necklace splitting theorem for two thieves, a theorem about chessboard splittings using hyperplanes with fixed directions, and all known cases of Langerman's conjecture about bisections with $n$ hyperplanes.

Our main result also confirms an infinite number of previously unknown cases of the following conjecture of Takahashi and Sober\'on: \emph{For any $d+k-1$
measures in $\rr^d$, there exist an arrangement of $k$ parallel hyperplanes that bisects each of the measures.}

The general result follows from the case of measures that are supported on a finite set with an odd number of points.  The proof for this case is inspired by ideas of differential and algebraic topology, but it is a completely elementary parity argument. 

Additionally, we disprove the general case of Langerman's conjecture, showing that the conditions in our main result are sometimes necessary.


\end{abstract}

\maketitle

\section{Introduction}

Understanding how one can split families of measures on real spaces into pieces of the same size is a central topic in topological combinatorics \cites{RoldanPensado2022, Zivaljevic2017}. The classic example, known as the ham sandwich theorem, states that \textit{given $d$ smooth probability measures on $\rr^d$, there exists a hyperplane simultaneously splitting each measure in half.}

Early popularity of this theorem was likely due to the elegance of its proof and maybe to the culinary reference. Before Stone and Tukey   called it the \emph{ham sandwich} theorem (as we slice a ham and two pieces of bread simultaneously) Steinhaus used the motivation of a \emph{leg of pork} (as we slice fat, bone, and meat).

A myriad of related results have appeared since then. Many of them recover the foodie references (cakes  \cite{cccg98-akiyama-perfect}, spicy chickens \cite{Karasev:2014gi}, pizzas \cite{Barba:2019to}, etc.), and most are proven with topological methods, which range from fixed point and Borsuk--Ulam type theorems to more advanced techniques in algebraic topology (cohomological index theory, equivariant obstruction theory, spectral sequences, etc).  Even splitting measures with families of few hyperplanes can lead to very rich topological statements (see, for instance, the history of the Hadwiger--Gr\"unbaum--Ramos theorem \cite{Blagojevic:2018jc}).

Beyond their inherent interest, these results have numerous applications in various fields like  computational geometry \cite{Matousek:1994kq}, combinatorial geometry \cites{alon2005crossing, Fox:2012ks} convexity \cites{lehec2009yao, FHMRA19}, geometric analysis \cites{Gromov:2003ga, Guth.2010}, harmonic analysis \cite{Guth:2016cy} and incidence geometry \cite{Guth:2015tu}. In these applications one may use partitions in which one part avoids a hyperplane \cite{yao1985general}, partitions in which continuous functions on convex bodies are equalized \cites{Gromov:2003ga,FHMRA19}, or partitions by the zero set of a polynomial \cite{Stone:1942hu}. The current paper extends this last type of result to polynomials that are products of affine functions satisfying certain constraints. As in the classical polynomial ham sandwich theorem, the Veronese trick combined with our theorem implies nonlinear analogues as corollaries.

Results which deal with partitions of $\rr^d$ into two sets separated by a possibly singular $(d-1)$-dimensional manifold  are known as ``chessboard colorings''. The most studied family of examples of chessboard colorings are the ones that have a hyperplane arrangement as boundary. Depending on the number of hyperplanes used and conditions on their position, there have been several important results in the area.  The backbone of the proof of each of these results is topological, most boiling down to the study of zeros of certain equivariant maps (the group involved changes from theorem to theorem).

\subsection{Our result}

The main result of this manuscript is a common generalization of several known results regarding chessboard colorings bounded by hyperplane arrangements.  A surprising aspect of the proof is that, rather than proving a common generalization to a scattered set of Borsuk--Ulam-type theorems, our elementary proof was inspired by one of the proofs of the Borsuk--Ulam theorem.  The main idea is to study the parity of a discrete set of partitions as we go through some geometric deformation of the measures.  The idea behind this approach, in the language of differential topology was used by Hubard and Karasev to confirm new cases of a conjecture of Langerman \cite{Hubard2020}. Later, Patrick Schnider managed to make the argument elementary in dimension two \cite{Schnider2021}.  We show how this elementary approach can be used in any dimension and involving families of parallel hyperplanes, which significantly increases the number of results we generalize.  Let us first describe the main theorem in full generality, and then specify how it implies several known results.

Given two positive integers $m,k$, let $S(m,k)$ be the \emph{Stirling number of the second kind}.  This counts the number of partitions of a set of $m$ elements into $k$ non-empty sets.  Given positive integers $m_1, \ldots, m_n$ and $M = m_1 + \dots + m_n$, let $\binom{M}{m_1,\dots, m_n}$ be the multinomial coefficient that counts the number of ordered partitions of a set of $M$ elements into $n$ sets $U_1, \dots, U_n$ such that $|U_i|=m_i$ for each $i$.

An oriented hyperplane $h$ in $\rr^d$ defines two closed halfspaces: $h^+$ and $h^-$, which we refer to respectively as the positive and negative halfspace of $h$.  A finite family $\mathcal{F}$ of oriented hyperplanes in $\rr^d$ defines a chessboard coloring, which is a pair of sets $(A,B)$ defined as:
\begin{align*}
A & = \{x \in \rr^d: x \mbox{ is in an even number of positive halfspaces of hyperplanes of }\mathcal{F}\}   \\
B & = \{x \in \rr^d: x \mbox{ is in an odd number of positive halfspaces of hyperplanes of }\mathcal{F}\}.
\end{align*}


Notice that since a hyperplane is the zero set of an affine function, $a\colon \rr^d \to \rr$, given a family of hyperplanes, with corresponding affine functions $\{a_1,a_2, \ldots a_M\}$ their chessboard coloring can alternatively be succinctly described by \[A=\{x \in \rr^d:p(x)\geq 0\} \textrm { and }B=\{x \in \rr^d:p(x)\leq 0\}, \textrm{ where, }\] 
\[p(x):=\prod_{i=1}^M a_i(x).\]

Given a hyperplane arrangement, with corresponding chessboard partition $(A,B)$, and a finite Borel measure $\mu$, we say that the arrangement {\bf bisects} the measure if 

\[
\min(\mu (A), \mu(B)) \geq \frac{\mu(\rr^d)}{2}. 
    \]

A central instance of this definition is when a finite measure $\mu$ is absolutely continuous with respect to the Lebesgue measure (or more generally, $\mu(h)=0$ for every affine hyperplane $h$), in this case the bisecting condition is simply:
\[\mu(A)=\mu(B).\]


   
 
\begin{theorem}\label{thm:main}
    Let $n,d$ be positive integers, and $L_1,L_2\dots, L_n$ be subspaces of $\rr^d$ of positive dimension. Let $l_i = \dim L_i$, and $k_1, \dots, k_n$ be positive integers. Let $G$ be the subgroup of permutations of $[n]=\{1,\dots,n\}$ such that for every $g \in G$, $L_{g(i)}=L_i$ and $k_{g(i)}=k_i$ for all $i$. Put $m_i:=l_i + k_i - 1$, $M: = \sum_{i=1}^n m_i$.
    Suppose that
    
    \[N:=\frac{1}{|G|}
    \binom{M}{m_1,m_2 \dots, m_n} \prod_{i=1}^n S(m_i, k_i) \equiv 1 \mod 2.
    \]
     Then, for any $M$ finite measures on $\rr^d$, 
     there exist unit vectors $v_1 \in L_1, v_2 \in L_2 \dots$, $v_n \in L_n$, and families of hyperplanes $\mathcal{F}_1,\dots, \mathcal{F}_n$ such that for each $i$, each hyperplane in $\mathcal F_i$ is orthogonal to $v_i$, $|\mathcal F_i|=k_i$, and the chessboard coloring induced by the family $\bigcup_{i=1}^n \mathcal{F}_i$ bisects each of the $M$ measures.
    
\end{theorem}
In one case this condition was known to be necessary for the theorem to hold.  The case is $n=2, l_1 = l_2 = 1, k_1 = k_2 = 1$, and a counterexample was constructed by Karasev et al. \cite{Karasev:2016cn}. In \cref{sec:counterexamples} we extend this known counterexample to many instances with $l_i=1$ for all $i$.  We also construct counterexamples to a conjecture of Langerman \cite{Barba:2019to}, described below, corresponding to the case $(k_i,l_i)=(1,d)$ for all $i \in [n]$.

\begin{question} Is the numerical condition of \cref{thm:main} necessary. 
\end{question}
 
We now make a few remarks to shed some light on this condition.  Firstly, the group $G$ is a direct product of symmetric groups. 
 A subset $I \subset \{1,2,\ldots n\}$ gives rise to a factor $S_{I}$, if and only if, for all $i,j \in I$, 
$L_i=L_j$ and $k_i=k_j$. 
Secondly, the multinomial coefficient $\binom{M}{m_1,m_2 \dots, m_n}$ is odd when the expressions of the numbers $\{m_1,m_2, \ldots m_n\}$ written in base two don't share any nonzero term 
(for instance if for each $i$, $m_i=2^i$, then the multinomial coefficient is odd, in contrast if e.g. $m_1=1,m_2=3$, then it is even, since in binary $m_1=01$ and $m_2=11$). 
Finally, it is known that:
\[
S(m, k)\equiv
\binom{m - \left\lceil \frac{k+1}{2}\right\rceil} 
{ \left \lfloor \frac{k-1}{2} \right \rfloor} \mod 2.
\]

Combining the previous remarks one can precisely understand in which cases $N$ is odd. Let us specify some instances of the theorem above that give interesting known results.  

\begin{itemize}
	\item The case $n=1, L_1 = \rr^d, k_1=1$ is the ham sandwich theorem \cite{Steinhaus1938}.  Since $S(d,1)=1$, we don't impose any additional conditions. 
 	\item The case $d=1$, $n=1$ and any $k$ is the ``necklace splitting theorem''.  This was originally proved by Hobby and Rice \cite{Hobby:1965bh} and then rediscovered in the 80's with two new proofs \cites{Goldberg:1985jr, Alon:1985cy}.  As $S(k,k)=1$, we don't impose any new requirement on $k$.
	\item The case $n=1, L_1 = \rr^d, k_1=2$ is a recent result of Sober\'on and Takahashi \cite{Soberon2021a}, which says that \textit{for any $d+1$ measures in $\rr^d$, there are two parallel hyperplanes whose chessboard coloring bisects each measure}. Notice that since $S(d+1,2)$ is always odd for $d \ge 1$, we don't impose additional conditions on $d$.
	\item The case $l_1 = \dots = l_n = 1$ is a result of Karasev, Rold\'an-Pensado, and Sober\'on \cite{Karasev:2016cn}, which describes chessboard colorings by families of hyperplanes with fixed directions.  The condition of parity of a multinomial coefficient is the same in the theorem above as the one found by Karasev et al.
	\item The case $L_1 = \dots = L_n = \rr^d$ and $k_1 = \dots = k_n=1$ is a recent result of Karasev and Hubard \cite{Hubard2020} which solved an infinite number of instances of a conjecture of Langerman \cite{Barba:2019to} which we recall below.  The condition of parity of the multinomial coefficient boils down to the dimension being a power of $2$, like in the aforementioned paper \cite{Hubard2020}.  Blagojevi\'c et al gave a different proof of the same statement \cite{Blagojevic2022}. 
\end{itemize}

Here is Langerman's conjecture.
\begin{conjecture}
For any $dn$ measures in $\rr^d$ there exist $n$ hyperplanes whose chessboard coloring bisects each of them. 
\end{conjecture}

We provide a counterexample to Langerman's conjecture for the case $(d,n)=(3,2)$.

\begin{theorem} There exists a family of six measures in $\mathbb R^3$ that cannot be bisected by two hyperplanes. 
\end{theorem}

The case $n=1, L_1 = \rr^d$ and any $k$ confirms an infinite number of cases of the conjecture below, due to Takahashi and Sober\'on \cite{Soberon2021a}. This is one of many generalizations of the necklace splitting theorem to higher dimensions, see \cites{Longueville2008, Blagojevic:2018gt} for other generalizations.

\begin{conjecture}\label{conj:takahashi}
For any $d+k-1$ measures in $\rr^d$, there exist $k$ parallel hyperplanes whose induced chessboard coloring simultaneously bisects every measure.    
\end{conjecture}

 We confirm \cref{conj:takahashi} when $S(d+k-1, k)$ is odd. Our theorem provides for each $d$ an infinite set of positive integers $k$, and for each $k$ an infinite number of $d$s, for which the pair $(d,k)$ satisfies this conjecture. In fact, if we arrange the pairs $(d,k)$ for which \cref{thm:main} implies the Takahashi--Sober\'on conjecture in a triangular array we obtain a Sierpinski triangle pattern.  

\subsection{Veronese trick}

Since the regions are bounded by hyperplanes, we can replace them by sets defined as zeroes of a polynomial.  This is better illustrated with an example, as in the following corollary, in which we consider a line as a degenerate case of a circle with center at infinity.

\begin{corollary}
    Let $\mu_1, \dots, \mu_7$ be seven finite measures on $\rr^2$, each absolutely continuous with respect to the Lebesgue measure.  There exist two concentric circles, a line and a vertical line whose union induces a chessboard coloring that bisects each measure.
\end{corollary}

\begin{proof}
    Consider the map $(x,y) \to (x,y,x^2+y^2)$.  This lifts each of the measures to a regular paraboloid $P$ in $\rr^3$.  Moreover, if $H$ is a plane in $\rr^3$, the projection of $H \cap P$ onto the $xy$-plane is a circle.  If $H$ has a normal vector in the $xy$-plane, then $H \cap P$ projects down to a line.  Apply \cref{thm:main} to the lifted measures using $n=3$ and:
    \begin{itemize}
        \item $L_1 = \rr^3, k_1 = 2$ (these two parallel planes will give the two concentric circles)
        \item $L_2$ equal to the the $xy$-plane, $k_2 = 1$ (this will give us the line without conditions), and
        \item $L_3 = \operatorname{span}\{(1,0,0)\}$, $k_3 = 1$.
    \end{itemize}

    Since $\binom{7}{4,2,1}=105$, $S(4,2)$ is odd, and $S(2,1)=S(1,1)=1$, we meet the conditions of \cref{thm:main}.
\end{proof}

The idea of using this trick for ham sandwich type results is already implicit in the seminal paper \cite{Stone:1942hu} of Stone and Tukey.  We can combine the Veronese trick with careful choices for the spaces $L_1$ to get chessboard partitions induced by curves of seemingly very different polynomials.

\subsection{Notation}
\label{notation}

To shorten some statements in the rest of the paper we use the notation of \cref{thm:main} throughout. We also use the following notation and terms:
\begin{itemize}
       
    \item $k:=(k_1,\dots, k_n),l:=(l_1, \dots, l_n)$ are positive integer vectors with $n$ entries each (in a lemma we take $n=1$).

     \item $B_{L_1},\dots, B_{L_n}$ are basis of the orthogonal complements $L_1^\perp,\dots, L_n^\perp$

    \item $\mathcal H=\cup _{i=1}^n\mathcal H_i$, where for each $i$, $\mathcal H_i:=\{h_{i,1}, h_{i,2}\ldots h_{i,k_i}\}$ is an arrangement of $k_i$ parallel hyperplanes. 
    
    \item $\mu^-=\{\mu_1,\mu_2, \dots, \mu_M\}$ is a family measures on $\rr^d$.
\end{itemize}

\subsection{Summary}
The proof of \cref{thm:main} has three steps. First, we construct a particular family of measures supported in sets of an odd number of points for which the number of  bisecting arrangements is exactly $N$. 
Second, we show that for any deformation of these point sets the parity of the number of bisecting hyperplane arrangements is invariant. %

Finally, we show that the result for measures supported on an odd set of points implies the general result.

\section{The ham sandwich theorem revisited}

In this section we show how our method gives a new and elementary proof of the ham sandwich theorem.  This is a particular case of the general proof described in the next sections, but may help the reader build intuition about this approach.  Due to the equivalence of the Borsuk--Ulam theorem and the ham sandwich theorem \cite{Karthik2017}, this can be used as one of the ingredients of a new convoluted proof of the Borsuk--Ulam theorem.

First, we state the version of the ham sandwich we will prove.

\begin{theorem}
    If $U_1, \dots, U_d$ are $d$ sets of points in $\rr^d$, each of odd cardinality and such that their union is in general position.  Then, there exist points $p_1 \in U_1, \dots, p_d \in U_d$ such that the hyperplane spanned by $p_1,\dots, p_d$ halves each of the sets $U_1,\dots, U_d$.
\end{theorem}

\begin{proof}
    We will show that the number of such halving hyperplanes is odd, so it cannot be zero. It is not difficult to show (see \cite{Barany:2008vv}) that if the sets $M_1, \dots, M_d$ are well separated (the convex hull of the union of any collection is  disjoint from the convex hull of the union of the rest), then there is exactly one halving hyperplane.  We consider each of $U_1, \dots, U_d$ as a different color.

 Suppose that we have another configuration of points $U'_1, U'_2, \dots, U'_d$ such that $|U_i|=|U'_i|$ for each $i$.  Consider a bijection between $U_i$ and $U'_i$ for each $i$.  We are going to continuously move the points of $(U_1, \dots, U_d)$ to their corresponding points in $(U'_1, \dots, U'_d)$.  For each point $p \in \bigcup_i U_i$, this gives us a function $p:[0,1] \to \rr^d$ such that $p(0) = p$ and $p(1)$ is the corresponding point in $\bigcup_i U'_i$.  Let $(U_1(t), \dots, U_d(t))$ be the induced configurations of points at time $t$.  We may assume without loss of generality that only for a finite number of exceptional times the configuration is not in general position, and when the configuration is not in general position, there is a unique $(d+1)$-tuple of points contained in a hyperplane.

    Each halving hyperplane at a time when $(U_1(t), \dots, U_d(t))$ is in general position goes through one point of each color.  For each colorful $d$-tuple $X(t)$ at time $t$, let $h_X(t)$ be the hyperplane they span.   Consider a colorful $d$-tuple $A(t)$, corresponding to a halving hyperplane.  If $A$ stops being a halving hyperplane at time $t_0$, it means that there is an $\varepsilon>0$ such that $h_A(t_0-\varepsilon)$ is a halving hyperplane, and $h_A(t_0+\varepsilon)$ is not.  Moreover, $h_A(t_0)$ must contain a point $p$ of the configuration that is not in $A$.  Suppose that the point is of color $i$, and let $B$ be the colorful $d$-tuple formed by replacing the point $p_i$ of color $i$ from $A$ by $p$.  Note that $h_B(t_0) = h_A(t_0)$, so it halves all colors except $i$, which is almost balanced.  Moreover, $p_i$ is on different sides of $h_B(t)$ for $t=t_0-\varepsilon$ and $t=t_0+\varepsilon$, so exactly at one of those times $h_B(t)$ is a halving hyperplane.

    If $h_B(t+\varepsilon)$ is a halving hyperplane, then the total number of halving hyperplanes did not change.  If $h_B(t-\varepsilon)$ is a halving hyperplane, then the total number of halving hyperplanes decreased exactly by two.  An analogous analysis shows that when a hyperplane starts being a halving hyperplane, then the total number of halving hyperplanes either stays constant or increases by two.  Therefore, the parity of halving hyperplanes does not change.  Since there is a configuration with exactly one halving hyperplane, we are done.
\end{proof}

 \begin{center}
 \begin{figure}\label{ham-sandwich}
    \begin{tikzpicture}
    
    \draw [blue, thick, dashed] (-1.5,-2) -- (3,2);
    \draw (1.5,2/3) node[red, thick]{\textbullet};
    \draw (1.9,2/3) node[red]{$r_1$};
    \draw (.75,0) node[red, thick]{$\circ$};
    \draw (1.10,0) node [red, right] {$r(t)$};
    \draw (.4,0) node[red]{$\circ$};
    \draw (1,1.5) node[red, thick]{\textbullet};
    \draw (-1.3,1.2) node[red, thick]{\textbullet};
    \draw (-.5,-.7) node[red, thick]{\textbullet};
    \draw (.7,-.4) node[red, thick]{\textbullet};
    \draw (-.4,-1.4) node[red, thick]{\textbullet};
    \draw (1.5,-1) node[red, thick]{\textbullet};
    \draw (2.3,.3) node[red, thick]{\textbullet};
    
    \draw (.5,0) edge [->] (.65,0);
     \draw (.85,0) edge [->] (1,0);
    \draw [brown, thick] (-1.1,-4/3)--(0,-2/3) -- (1.1,0)--(3.3,4/3);
    \draw (1.1,0) node[red]{$\circ$};
    \draw (0,-2/3) node[black, thick]{\textbullet};
    \draw (0.2,-2/3-0.1) node[black]{$b$};
    \draw (0,2/3) node[black, thick]{\textbullet};
    \draw (.2,-1) node[black, thick]{\textbullet};
   
    \end{tikzpicture}
\caption{Here $U(t)=U_1(t) \cup U_2(t) \subset \rr^d$, represented by red and black points. The red hollowed point $r(t)$ moves right along the arrowed line. It is represented at three different times. Before the movement the dashed blue line $h$ containing $r_1$ and $b$ bisects the point sets. When $r(t)$ is incident to $h$, the line $h'(t)$ spanned by $r(t)$ and $b$ is almost bisecting and coincides with $h$. When $r(t)$ passes to the right of $h$, the line $h'(t)$ becomes bisecting, and $h$ stops being bisecting.}
    \end{figure}
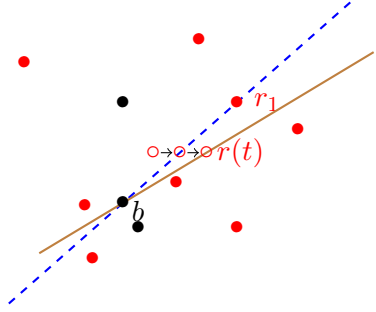
\end{center}

Once we involve families of parallel hyperplanes, there will be more nuance for the parity analysis.  The key ideas remain the same.

\section{Proof}
We will say that a measure is {\bf oddly supported} if it is the sum of an odd number of delta masses. Generic families of oddly supported measures are central to our approach, because in this case, one of the points of each measure lies in the hyperplane arrangement, and each of the open chessboard regions contains exactly half of the remaining points. In consequence, the number of bisecting arrangements for generic oddly supported measures is finite.  

\subsection{Points in generic position}

We say that a finite set of points $S$ in $\rr^d$ is in {\bf generic position} if the set of all $d|S|$ coordinates of the points is algebraically independent. 
In other words, the only multinomial with integer coefficients that has a zero if we evaluate it using only values of the $d|S|$ coordinates is the constant zero.  
In particular, if $S$ is in generic position any $d$ vectors formed by the differences of $d$ different pairs of points of $S$ that don't form any cycles is linearly independent.  
Notice that this condition is much stronger than general position, in which we require that no $d+1$ points of $S$ are contained in a hyperplane. 
Indeed, if we have $x_0, \dots, x_{d}$ points in a hyperplane, the $d$ vectors of the form $x_i - x_0$ would be linearly dependent.
It is easy to show that if $\{x_1,x_2, \ldots x_m\}$ is a point set and $\{\xi_1,\xi_2, \ldots \xi_m\}$ is a set of independent random vectors, chosen uniformly from the ball $B(0,\delta)$ of radius $\delta>0$, then $\{x_1+\xi_1,x_2+\xi_2, \ldots x_m+\xi_m\}$ is in generic position almost surely.

Whenever we have a partition of a set, its {\bf core} consists of all the subsets of the partition that have more than one element. In the following discussion, we intentionally use language loosely and conflate a partition of a point set with the corresponding partition of the finite set that indexes the point set.


\begin{lemma}\label{lem:conteo-chico}
Let $L$ be a subspace of $\rr^d$ of dimension $l$, $B_L$ a basis of $L^{\perp}$, and $P$ a set of $m=l-1+k$ points such that $P \cup B_{L}$ is in generic position in $\rr^d$.
The set of $k$-tuples of parallel hyperplanes whose union contains $P$ and whose normal vector is contained in $L$ is in bijection with the partitions of $P$ into $k$ non-empty sets.
\end{lemma}

\begin{proof}

Suppose that $h_1, \dots, h_k$ are parallel hyperplanes whose union contains $P$ and whose normal vector $v$ is contained in $L$. We claim that none of the sets of the partition
$\chi=(h_1 \cap P, \dots, h_k \cap P)$ is empty. Let $r$ be the number of non-empty sets in this partition. For each $i \leq k$ such that $h_i \cap P$ is in $\chi$ choose some element $p_i \in h_i \cap P$. Put $k_i:=|h_i \cap P|$. 
Let $V$ be the union $\cup_i V_i$, where $V_i$ is the set of vectors of the form $q-p_i$ with $q \in h_i \cap (P\setminus \{p_i\})$. Then 
\[|V|=\sum_{j=1}^r{(k_j-1)} = (l+k-1)-r.\]

If $r \le k-1$, this means that we have constructed a set $V$ of at least $l$ vector differences, contained in $v^{\perp}$. Since $B_L$ contains $d-l$ vectors, and $B_L \cup V \subset v^{\perp}$, then the set $B_L \cup V$ is linearly dependent, contradicting the generic position assumption. Hence $r\geq k$, showing that each hyperplane contains at least one point. 

Now, consider a partition $\chi$ of $P$ into $k$ non-empty sets $P_1, \dots, P_k$.  For any $P_i$, consider a point $p_i \in P_i$, and then take the $|P_i|-1$ vectors $q-p_i$ for $q \in P_i\setminus \{p_i\}$.  
We have formed a total of $\sum_{i=1}^k (|P_i|-1)= l-1$ vectors.  
These vectors together with $B_L$ are linearly independent, so they form a basis of a hyperplane $h$.  
Then, consider $h_i$ the translate of $h$ through $p_i$. 
By construction, $P_i \subset h_i \cap P$.  By the arguments above, it is not possible to cover $P$ with fewer than $k$ translates of $h$, so $P_i = h_i \cap P$.  This means that for every partition of $P$ into $k$ non-empty subsets, there exists a $k$-tuple of translates of a hyperplane that induces that partition with perpendicular vector in $L$. 

In fact the two functions we described from partitions to arrangements, and from arrangements to partitions are inverses of each other.
\end{proof}

Before generalizing this lemma to arrangements with several sub-arrangements of parallel hyperplanes, we prove a similar lemma that will be used in \cref{deformation}.

\begin{lemma}\label{circulo}
Let $L$ be a subspace of $\rr^d$ of dimension $l$, $B_L$ a basis of $L^{\perp}$, and $P$ a set of $m=l-2+k$ points such that $P \cup B_{L}$ is in generic position in $\rr^d$. Let $X$ be a fixed partition of $P$ into $k$ non-empty subsets. There is a unique subspace $K$ of dimension $d-2$ such that there exists a vector $v \in L$ orthogonal to $K$ and there are $k$ translates of $K$ that cover $P$ inducing $X$ as a partition of $P$.
\end{lemma}

\begin{proof}
    The proof is analogous to the one for \cref{lem:conteo-chico}.  Let $X = X_1 \cup \dots \cup X_k$.  Pick a point $p_i \in X_i$ for each $i =1,\dots, k$, and consider the $l-2$ vectors of the form $p - p_i$ for $p \in X_i \setminus \{p_i\}$.  These, together with $B_L$, form $d-2$ linearly independent vectors, which gives us the basis for $K$.  The $k$ translates of $K$ containing each of $p_1, \dots, p_k$ are the affine spaces we were looking for.
\end{proof}

Let $l$ and $k$ be two positive integer vectors of length $n$. A {\bf valid $(l,k)$-partition} of a finite set $[M]$, consists of a labeled partition of $[M]$ into $n$  subsets, each of which is refined to an unlabeled partition of non-empty subsets. 
The $i$-th set of the first partition must have cardinality $l_i-1+k_i$, which is refined into $k_i$ non-empty subsets in the second partition. \\

For example if $M=7$, and $l=(2,2)$, $k=(3,2)$, then a valid $(l,k)$ partition of the set $[7]=\{1,2,3,4,5,6,7\}$, can be represented by
\[14|3|7 || 25|6,\]

where the first ordered partition is 
$1347|256$ and the subset $1347$ is partitioned to $14|3|7$, while the set $256$ to $25|6$. As a valid partition, it is equivalent to  
\[3|41|7||6|52,\]

since the order of the refined partition is irrelevant, but different from,

\[25|6||14|3|7.\]
since the order of the first partition is relevant. 

Given a point set $P$ in $\rr^d$ with cardinality $M$, we say that a hyperplane arrangement $\mathcal H:=\mathcal H_1 \cup \mathcal H_2 \cup \ldots \cup \mathcal H_n$ is {\bf $(\mathcal L,k,P)$-valid}, if each hyperplane contains at least one point of $P$, $\mathcal{H}_i$ consists of $k_i$ parallel hyperplanes all of which are perpendicular to some $v_i \in L_i$.  We also say that two such arrangements $\mathcal H:=\mathcal H_1 \cup \mathcal H_2 \cup \ldots \cup \mathcal H_n$ and $\mathcal H':=\mathcal H'_1 \cup \mathcal H'_2 \cup \ldots \cup \mathcal{H}'_n$ are equivalent if there is a permutation $g:[n]\to[n]$ such that $\mathcal{H}_i' = \mathcal{H}_{g(i)}$, $L_{g(i)}= L_i$, and $k_{g(i)}=k_i$ for all $i \in [n]$.  

\begin{lemma} \label{lem:conteo-grande}  Let $n,d$ be positive integers, and $L_1,L_2\dots, L_n$ be subspaces of $\rr^d$ of positive dimension. Let $l_i = \dim L_i$, and $k_1, \dots, k_n$ be positive integers. Let $G$ be the subgroup of permutations of $[n]=\{1,\dots,n\}$ such that for every $g \in G$, $L_{g(i)}=L_i$ and $k_{g(i)}=k_i$ for all $i$. Put $m_i:=l_i + k_i - 1$, $M: = \sum_{i=1}^n m_i$.
Let $P:=\{p_1,p_2\ldots p_M\}$ be a set of points, such that $P \cup B_{L_i}$ is generic for $i=1,\dots,n$.  For every $(l,k)$-valid partition $\chi$ of $P$, there exists a unique $(\mathcal L,k,P)$-valid hyperplane arrangement $\mathcal F=\mathcal{F}_1 \cup \dots \cup \mathcal{F}_n$ such that the induced refined partition $\{h \cap P: h \in \mathcal F\}$ is the valid partition $\chi$.  Additionally, the number of equivalence classes of $(\mathcal L, k, P)$-valid arrangements is $N$.
\end{lemma}

\begin{proof}
    By \cref{lem:conteo-chico}, if a hyperplane arrangement is $(\mathcal L,k,P)$-valid, the condition on each $\mathcal H_i$ implies that $|P \cap \mathcal H_i| = l_i + k_i-1$ for all $i$.  
    So $\mathcal H$ corresponds to an $(l,k)$-valid partition of $P$: we must first find an ordered partition of $P$ into sets of size $l_1+k_1-1, \dots, l_n +k_n-1$ to determine the sets $\mathcal H_i \cap P$.  Then, to find the hyperplanes of $\mathcal H_i \cap P$ we must find a partition of $\mathcal H_i$ into $k_i$ non-empty parts.  

    Notice that for each maximal subset of indices $I$ such that for all $i,j \in I$, $L_i=L_j$ and $k_i=k_j$, we are over-counting the corresponding arrangements $|I|!$ times.  The number of equivalence classes follows directly.

\end{proof}

In what follows, $l,k$, $\mathcal {L}$ and $P$ are always fixed in advance, so to unclutter notation we just talk about {\bf valid partitions}, and {\bf valid arrangements} but we insist that we are actually talking about a partition and a refinement of it, on the one hand, and about a hyperplane arrangement with certain constraints on their directions coming from $\mathcal L$ and $k$, on the other.

\subsection{Well separated families of small segments}

Now we construct a particular family of generic oddly supported measures.  Suppose we are given the family $\mathcal{L}$ of subspaces and the vectors $k,l$. We say that a family of sets $\{K_1, K_2, \ldots K_M\}$ in $\rr^d$ for $M>d$ is {\bf well separated} if for each $i=1,\dots, n$ a set of $k_i$ parallel hyperplanes, whose normal vector is in $L_i$, can intersect at most $k_i + l_i - 1$ sets of the form $\conv(K_j)$.

\begin{lemma}\label{lem:conteo-arreglado} Let $\mathcal L$ be an ordered set of $n$ sub-spaces of $\rr^d$, each of positive dimension. Put $l_i=\dim(L_i)$, $(k_1, k_2 , \ldots k_n) \in \mathbb N^n$.  For each $j \in [M]$, let $q_j$ be an odd positive integer. There exists $\mathcal U$ a well separated $M$-colored point set, such that for each $j \in [M]$, $U_j$ has $q_j$ points, for each $i \in [n]$, $(\cup_j U_j) \cup  (\cup B_{L_i})$ is generic and such that the set of labeled valid arrangements that bisect each $U_j$ is in bijection with the set of $(l,k)$-valid partitions of $[M]$.
\end{lemma}

\begin{proof}
Let $P$ be a set of $M$ points such that for each $i$, $P \cup B_{L_i}$ is in generic position.

Consider the set $S$ of  $(\mathcal{L}, k, P)$-valid hyperplane arrangements.  For $p_j \in P$, let $\ell_j \in S^{d-1}$ be a direction that is not contained in any hyperplane of $S$.  Due to \cref{lem:conteo-grande}, $S$ is finite, so we can choose the directions $\ell_1,\dots, \ell_M$ that are in general position.

Now, for $U_j$ take $q_j$ points in a segment with direction $\ell_j$ such that $p_j$ is the midpoint.  We can take this segment small enough so that for any hyperplane arrangement in $S$, exactly one hyperplane intersects it.

Now, for any choice $P'=(p'_1,\dots,p'_M)$ with $p'_j \in U_j$ for all $j \in [M]$, consider a valid $(\mathcal{L}, k, P')$-arrangement.  Since each $U_j$ is small enough, we can assume that only one hyperplane in it cuts the segment $U_j$.  Therefore, to bisect $U_j$ we need $p'_j = p_j$.  This means that the only valid hyperplane arrangements that bisect each $U_j$ are precisely the $(\mathcal{L},k,P)$-valid hyperplane arrangements, as desired.

A small perturbation that puts the entire set $U \cup (\cup_i B_{L_i})$ in generic position preserves this count.

\end{proof}

In particular, the number of equivalence classes of valid arrangements that bisect each set of points is $N$.

\subsection{Parity under deformation}\label{deformation}

This section contains the key step in the proof of \cref{thm:main}.  We refer to a smooth family of colored point sets parameterized by $[0,1]$ as a {\bf path of point sets}, which we denote by $U(t):=\{ U_i(t)\}_{i=1}^M$. As before, we assume that each color class has an odd number of (paths of) points. We say that a path of point sets is {\bf generic} with respect to an $n$-tuple of sets of vectors $\{B_{L_i}\}_{i \in [n]}$ if for at most a finite set of times $t$, which we call {\bf exceptional times} there is a unique $i$, such that $U(t) \cup B_{L_i}$ is not generic. Moreover we assume that at an exceptional time $t_j$, the point set $U(t_j)$ is {\bf almost generic}, in the following sense.  There exists a colorful set $S$ of $M$ points and a valid partition $P$ of $S$ such that the hyperplane arrangement they induce contains exactly $M+1$ points.  Moreover, this set of $M+1$ points and the induced partition are uniquely determined, in the sense that there exist a unique index $i\leq n$, and a set of vectors $V(t_j) \subset U(t_j)$, with $|V(t_j)|=l_i+k_i$, such that any subset $U'(t_j)\subset U(t_j)$, for which $U(t_j)' \cup B_{L_i}$ is non-generic must contain $V(t_j)$, and $V(t_j) \cup B_{L_i}$ which is itself non-generic.  We say that $\mathcal H_i$ is the \textbf{oversaturated} family of $\mathcal{H}$.

\begin{lemma}\label{parity} Let $\mathcal L$ be a family of $n$ subspaces, and $\mathcal U(t)$ a  generic path of odd point sets of $\rr^d$. The parity of the number of valid bisecting arrangements does not depend on $t$. 
\end{lemma}

 Assuming this lemma we are ready to give a proof of \cref{thm:main} for generic oddly supported measures. 

\begin{proof}[Proof of \cref{thm:main}] 
It is easy to see that any two generic families of point sets
$$\{U_1,U_2, \ldots U_M\}$$ and $$\{U_1',U_2', \ldots U_M'\}$$ such that $|U_i|=|U_i'|$ can be connected by a generic  path of point sets. 
One might achieve this by choosing arbitrary bijections and moving one point at the time on a path that avoids the $(d-2)$-flats that contain more than $d-1$ points of $U$. So if $U(0)$ is an arbitrary generic point set, with $M$ color classes, with an odd number of points in each, there exists a generic path of odd point sets $U(t)$ such that the colored point set $U(1)$ is generic and the color classes are very well separated. By \cref{parity} the number of valid bisecting arrangements of $U(0)$ has the same parity as the number of such arrangements of $U(1)$. By \cref{lem:conteo-grande} the latter equals $N$. In particular if $N$ is odd, at least one valid bisecting arrangement should exist for $U(0)$.
\end{proof}

 If at an exceptional time there exists a hyperplane arrangement $\mathcal H$, that is valid except that it intersects two points of the same color, we call it {\bf almost valid.} For non-exceptional times we defined a bijection that takes each valid partition of a colorful point set to a valid hyperplane arrangement. At exceptional times, this function is still well defined but cannot be inverted. When we refer to a path of valid arrangements $\mathcal H(t)$, we assume that the valid partition, and the colorful subset are fixed. 

 We will say that a valid arrangement is {\bf almost bisecting} if it is bisecting for every color class except for one that we call the {\bf unbalanced color}. This color class is assumed to contain two points that are incident to the arrangement and the number of points of this color in the interiors of the chessboard regions differ by exactly one.

\begin{claim}\label{almost_bisecting} Let $\mathcal L$ be a family of $n$ subspaces, and $\mathcal U(t)$ a  generic path of $M$-colored odd point sets of $\rr^d$, let $t_*$ be an exceptional time, so that $\mathcal H(t_*)$ is an almost bisecting arrangement of $U(t_*)$. If a point of the unbalanced color $q(t_*)$ is contained in $U(t_*) \cap h(t_*)$ for some $h(t_*)\in \mathcal H(t_*)$, but for $\epsilon>0$ small enough,  $q(t_*-\epsilon) \notin h(t_*-\epsilon)$, then exactly one of the arrangements $\mathcal H(t_*-\epsilon)$, $\mathcal H(t_*+\epsilon)$ is a valid bisecting arrangement. 
\end{claim}

\begin{proof}
Since the problem is local, we might parameterize in a neighbourhood around $t_*$ so that all the points of $U(t)$ are fixed except for $q(t)$ with a trajectory that is transversal to $h(t_*)$. From the definition of almost bisecting it follows that exactly one of the two directions balances the color of $q$
\end{proof}

Let $\mathcal H$ be a valid or almost valid arrangement, and $\mathcal H_i$ a sub-arrangement with common perpendicular direction $v$ in $L_i$. The {\bf core} of $\mathcal H_i$, is the set of points in $U$ that determine the direction $v$, in other words a point  in $U \cap \mathcal H_i$ is in the core, if it is contained in a hyperplane $h \in \mathcal H_i$, such that $|h \cap U|>1$. 

\begin{lemma}\label{moving} If $\mathcal U(t)$ is a  generic point set path, $t_*$ an exceptional time, and $\mathcal H(t_*)$ an almost bisecting arrangement, then there exists a unique valid arrangement path $\mathcal F(t)\neq \mathcal H(t)$, such that $\mathcal F(t_*)$ is almost bisecting, and the oversaturated family of $\mathcal F(t_*)$ has the same core as the oversaturated family of $\mathcal H(t_*)$.
\end{lemma}

Before going into the proof of this key lemma, let us show how combining it with \cref{almost_bisecting} we obtain \cref{parity}. 

\begin{proof}
Again we might think of $q(t)$ moving and the rest of the points being fixed. Since the path is generic any bisecting arrangement at time $t_j-\epsilon$ that is not almost bisecting at time $t_j$, is bisecting in an interval that contains the interval $[t_{j-1},t_{j+1}]$. So we only need to analyze the paths of arrangements that are almost bisecting at $t_j$. 
By \cref{moving} if there are any such arrangements, then there is a set of in disjoint pairs. Let $\mathcal H(t_j+\epsilon)$ and $\mathcal F(t_j+\epsilon)$ be one such pair. Now we use \cref{almost_bisecting} to analyse the four cases: if both are bisecting, the number of bisecting arrangements increases by two at $t_j$, if both are non-bisecting, then the number of bisecting arrangements decreases by two at $t_j$. If one of them is bisecting and the other one is not (this can happen in two ways), then the number of bisecting arrangements remians invariant at $t_j$. When we consider all such pairs, the parity constraints add up, and therefore the parity of the total number stays invariant.
\end{proof}

Our final step is proving \cref{moving}. 

\begin{proof}[Proof of \cref{moving}]

Assume that $\mathcal H(t_j)$ is almost bisecting.  We are going to stop the movement of the points, and move continuously $\mathcal H(t_j)$ along a path of arrangements until we find another valid almost bisecting arrangement, $\mathcal F(t_j)$. 

We denote the resulting path of valid almost bisecting hyperplane arrangements by $\mathcal A(s)$. So $\mathcal A(0):=\mathcal H(t_*)$, $\mathcal A(1):=\mathcal F(t_*)$ and for every $s \in [0,1]$, the point set $\mathcal U(t_*)$ is fixed.

 Denote by $(S_1,P_1)$ the colored set of points and valid partition that induced $\mathcal{H}(t_j-\varepsilon)$.  Find the unique hyperplane $h^1(t_*-\epsilon) \in \mathcal H(t_*-\epsilon)$ that contains the point $p_1$ of the unbalanced color not in $S_1$ at time $t_j$.  This is the only point contained by $\mathcal H(t_*)$ that is not in $S_1$. We denote by $S = S_1 \cap \{p_1\}$ the set of $M+1$ points contained in $H(t_*)$ and by $P$ the partition induced by $\mathcal{H}(t_*)$ on $S$, which consists of adding $p_1$ to one of the sets in the second layer of $P_1$.
 
 We call this $h^1$, the {\bf moving hyperplane}.  It will remain the moving hyperplane until we reach a point of $U(t_*)$ and we might switch moving hyperplane.  Denote by $H(t_*)$ the group of parallel hyperplanes that contains $h^1(t_*)$.  Note that $H(t_*)$ is the oversaturated family of $\mathcal{H}(t_*)$.
 
 Let $h(t_*)$ be the hyperplane of $\mathcal{H}(t_*)$ that contains the point $p_0$ of $S_1$ of the same color as $p_1$.  It might be that $h(t_*)=h^1(t_*)$, this is what happens in the ham sandwich theorem. In this case, we can put $\mathcal F(t_*)=\mathcal H(t_*)$. The arrangement $\mathcal F(t)$ will be induced by replacing $p_0$ by $p_1$ in both $S_1$ and $P_1$.  In general, if $p_0$ is in the core of $H(t_*)$, then we can simply replace $p_0$ by $p_1$ in both $S$ and $P$ to obtain $\mathcal{F}(t_*)$.  Otherwise, we will start moving $h(t_*)$, with the process described below, until we reach a new point of $U(t_*)$. 
 
We now explain how to define a path $\mathcal A(s)$ moving $h^1(s)$ (and sometimes the hyperplanes parallel to it) until we reach a new point that defines a new unbalanced color, which defines a new moving hyperplane $h^2(s)$. We repeat this process until we reach a point that is in the core of $H(t_*)$, at which point the process terminates.  We first describe the movement process, than then show that it always terminates.

We will define a sequence of hyperplanes $h^1, \dots, h^r, \dots$ and time intervals \[[s_1,s_2], \dots, [s_r, s_{r+1}]\] such that $h^r$ is the moving hyperplane exactly in $[s_r,s_{r+1}]$.  The path of arrangements $\mathcal A(s)$ is the concatenation of all these movements. Below, we explain how $h^r(s)$ moves when $s \in [s_r,s_{r+1}]$. We denote by $j(r)$ the color that is unbalanced at time $r$.  Throughout this process

 \textbf{If $h^{r-1}$ at time $s_{r}$ reaches a point whose color is not in the core of $H(t_*)$:} 

Let $h^{r}(s_r)$ be the hyperplane containing the other point of color $j(r)$
 
 \textbf{Case 1 : $h^r(s_{r})$ contains a single point.}

 In this case we translate $h^r(s)$ along its perpendicular direction. 
 Among the two opposite possible directions there is one along which the balance of color $j(r)$ is immediately fixed.
 We continue translating $h^r(s)$ in the same direction until we hit another point not covered by the arrangement. We are allowed to pass the translation through infinity and come back on the parallel hyperplane which re-enters the convex hull of the points from the opposite side. So eventually we hit a new point.  The new point defines the new unbalanced color $j(r+1)$,
 we let $h^{r+1}(s)$ be the hyperplane containing the other point of the new unbalanced color $j(r+1)$.  We update $S$ and $P$ and reiterate this procedure. See \cref{fig:translation}.

\textbf{Case 2 : $h^r(s_r)$ contains more than one point.}

 Let $p$ be the point of color $j(r)$ in $h^r(s_r)$.  Let $H'$ be the family of hyperplanes parallel to $h^r(s_r)$.  Note that $p$ is in the core of $H'$.  Since $p$ is not in the core of $H(t_*)$, this is a different group of parallel hyperplanes.  Denote by $L_{i(r)}$ the space that determines the possible normal vectors for $h^r(s_r)$, and $k_{i(r)}$ be the number of hyperplanes of $H'$.  Let $X$ be the set of points covered by $H'$.
 
 Consider the set $X\setminus\{p\}$.  By \cref{circulo}, there is a unique $(d-2)$-dimensional space $K$ and $k$ translates of $K$ that cover $X\setminus\{p\}$ inducing the same partition as $H'$ on them.  Let $K_1$ be the translate that covers $X \cap h^r(s_r)$.  Note that $K^{\perp}$ is a two-dimensional space.  Under the orthogonal projection to $K^{\perp}$, $X \setminus \{p\}$ is mapped to $k$ distinct points $p_1, \dots, p_k$.  The way we form $H'$ is by joining $p_1$ with the projection of $p$, forming a line $\ell$, and then take the inverse of the orthogonal projection of the $k$ translates of $\ell$ that contain each of $p_1, \dots, p_k$.

 To rotate $h^r$, we simply rotate $\ell$ around $p_1$, doing the same for its $k-1$ translates, and then take the inverse image.  As before, of the two possible directions to start the rotation, there is a unique direction that balances color $j(r)$.  We continue this rotation until the time $s_{r+1}$ that a hyperplane of $H'$ contains a point not previously contained by $\mathcal{A}(s_r)$ for $s>s_r$.  We update $P, S$ and repeat this process.  See \cref{fig:rotation} for an illustration.

This finishes the description of the algorithm to construct $\mathcal A(s)$. 

{\bf Correctness}

We now finish the proof of \cref{moving}. We focus on an exceptional time, so the points are not moving, only the hyperplanes. 
Notice that every time we switch the moving hyperplane, we obtain a new almost bisecting arrangement. We claim that the process terminates by eventually reaching a new point in a color of the core of $\mathcal H(T_*)$. As before we parametrized time in 

Assume the contrary in search for a contradiction.  The sequence of moving hyperplanes is finite, since they are all almost valid, we must eventually have to repeat a pair consisting of a configuration and an unbalanced color, forcing us into a cycle. But this is absurd since the path $\mathcal A(s)$ is one-dimensional and can be reversed all the way to $\mathcal H(t_*)$. There cannot be a cycle.
\end{proof}

\begin{figure}
    \centering
    \includegraphics{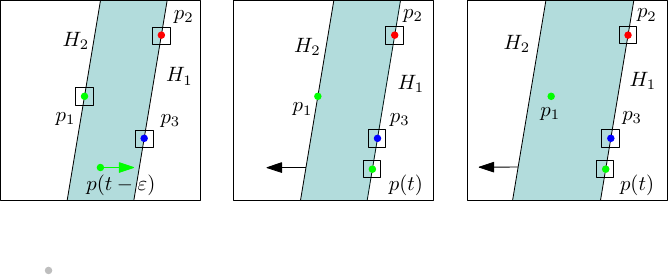}
    \caption{Example of a case when we a translation is needed.  In this case, $p$ belongs to $\supp \mu_1$ and will cross the hyperplane $H_2$ of the arrangement, which contains $p_2, p_3$.  If $p_1$, the point in the arrangement from $\supp \mu_1$, is contained in a hyperplanes $H_2$ that contains no other point of the arrangement, we need to translate $H_2$ until it hits a new point of $U(t)$ not already contained in $A(t)$.  The direction we translate $H_2$ is uniquely determined by which region lost $p(t-\varepsilon)$, as $p_1$ should replace $p(t-\varepsilon)$ in that region.\\
    This figure exemplifies the case when $H_2$ is in the class of hyperplanes parallel to $H_1$.  The same process follows if this is not the case.  The square boxes indicate the candidates for the points generating the new arrangement $B(t)$.  If the first point $H_2$ arrives to is in the support of $\mu_1(t), \mu_2(t), \mu_3(t)$, we would replace $p(t), p_2,$ or $p_3$ by it, as needed.  If it is in the support of another measure, further translations or rotations are needed to arrive to $B(t)$.}
    \label{fig:translation}
\end{figure}

\begin{figure}
    \centering
    \includegraphics{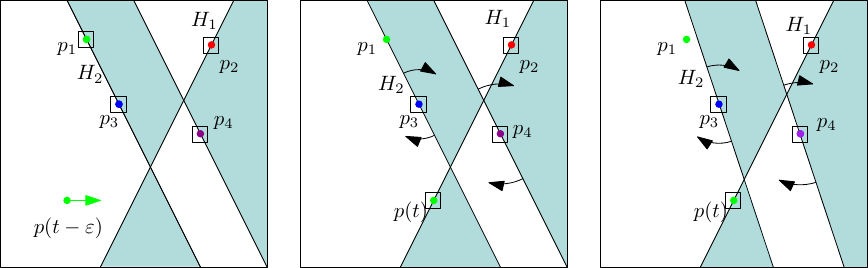}
    \caption{A similar setting as \cref{fig:translation}, but now $p_1$, the point of $A(t)$ in the same measure as $p(t)$ is a hyperplane which has additional point of the support ($p_3$ in the figure above).  In this case, we rotate the class of hyperplanes parallel to $H_2$.  The rotation of $H_2$ is around $p_3$, and every hyperplane parallel to $H_2$ has a unique rotation point.  The direction of the rotation is uniquely determined, as $p_1$ must be at time $t$ in the same region as $p(t-\varepsilon)$ was at time $t-\varepsilon$.  The rotation continues until the class of hyperplanes parallel to $H_2$ contain as new point of $U(t)$ that is not contained in $A(t)$.}
    \label{fig:rotation}
\end{figure}

\subsection{From oddly supported to general measures}
Our theorem follows for the particular case in which each 
of the measures is of the form $\frac{1}{|X|} \sum_{x \in X} \delta_x$, 
and $X \subset \rr^d$ is a finite set with an odd number of points. 
An advantage of this case, which we call an \emph{oddly supported measure},
is that every family of parallel hyperplanes that bisects a set with an odd number of points, must have one of the hyperplanes passing through one of the points. Any affine hyperplane is the zero set of an affine function of the form \[(x_1,x_2, \ldots x_d)\to \langle (x_1,x_2, \ldots x_d), (v_1,v_2, \ldots v_d) \rangle+v_{d+1},\]  where $(v_1,v_2, \ldots v_d,v_{d+1})\in \mathbb S^d$. The cases where the first $d$ coordinates of $v$ are zero correspond to the the empty set when $v_{d+1}=-1$ and the whole space when $v_{d+1}=1$.

Now for any Borel measure $\mu$, if $X$ is a sample of $2r+1$ points independently at random from $\mu$ and form the measure $\mu^r:=\frac{1}{|X|} \sum_{x \in X} \delta_x$, then, as $r$ tends to infinity, $\mu$-almost surely $\mu^r$ converges weakly to $\mu$. That is, for every closed set $C$, $\lim \sup_{r \to \infty}\mu^r(C)\leq\mu(C)$.
Assume that we have shown the theorem for oddly supported measures, let $A_r, B_r$ be the corresponding (closed) chessboard regions. 

\begin{lemma} Let $L_1,L_2, \ldots L_n$ be subspaces of $\mathcal \rr^d$, and $\mu$ be a Borel measure in $\rr^d$. Let $\mu^r$ be a sequence of measures converging weakly to $\mu$. If for each $r$, $(A_r, B_r)$ is a valid chessboard partition of $\mathbb R^d$ that bisects $\mu^r$, then there exists a subsequence of chessboard partitions that converges to a chessboard partition $(A,B)$ that bisects $\mu$.
\end{lemma}

\begin{proof}
The chessboard partition $(A_r, B_r)$ is determined by a valid hyperplane arrangement $\mathcal H^r$. Valid hyperplane arrangements are parametrised by a closed subset of $(\mathbb S^d)^{|\mathcal H^r|}$, which is a compact space. Hence, a subsequence of arrangements (and of chessboard regions) converges to a hyperplane arrangement $\mathcal H$.

Let $A$ and $B$ be the corresponding limiting chessboard regions. Then for any $\delta>0$ there exists $r_0$ large enough, so that for every $r>r_0$,
\[\mu(A)\geq \mu^r(A) - \delta \geq \mu^r(A_r) - \mu^r(A_r\setminus A)-\delta \geq\]
\[ \mu^r(A_r) - \mu(A_r\setminus A)-2\delta
\geq
\mu^r(A_r)-3\delta\geq 
\frac{1}{2}-3\delta,\]
hence $\mu(A)\geq \frac{1}{2}$,
and similarly for $B$. 
\end{proof}

To derive \cref{thm:main} as claimed, we apply this lemma to each of the measures $\mu_i$, refining the converging subsequence of cheesboard regions. Since we do this a finite number of times, the final subsequence converges. Notice that we are allowing for one or more hyperplanes to escape to infinity.

\section{Necessity of conditions}\label{sec:counterexamples}

\subsection{Hyperplanes with fixed directions}

We show that for the case $l_1 = \dots = l_{n}=1$, if $L_1, \dots, L_n$ are all different and at least two of the numbers $k_1, \dots, k_n$ are odd, then \cref{thm:main} may fail.  Note that the parity condition for \cref{thm:main} in that instance is that the mulitnomial coefficient
\[
\binom{M}{k_1, \dots, k_n}
\]
must be odd.  This happens if and only iff the numbers $k_1,\dots, k_n$ do not share $1$'s in the same position when written in base two.  If two of them are odd, then we are breaking this property in the first instance.

\begin{claim}
Let $l_1 = \dots = l_{n}=1$, and $L_1, \dots, L_n$ be $n$ different lines in $\rr^d$ through the origin.  Let $k_1, \dots, k_n$ such that at $k_1, k_2$ are odd.  Then, there is a set of $M=k_1 + \dots + k_n$ measures on $\rr^d$ for which no set of $M$ hyperplanes such that for each $i\in [n]$, exactly $k_i$ of the hyperplanes are orthogonal to $L_i$ for each $i \in  [n]$ induces a chessboard coloring that bisects each measure.
\end{claim}

\begin{proof}
    We will divide the proof into two cases, when $M$ is even and when $M$ is odd.

    \textbf{Case 1.} $M$ is even.

    In this case, consider a set of $M/2$ segments in $\rr^d$ such that none of them is parallel nor orthogonal to an $L_i$, and their midpoints are in general position.  Assume that the segments are sufficiently short so that no hyperplane orthogonal to an $L_i$ can intersect two of the constructed segments.

    Finally duplicate each segment and translate their copy orthogonally in a direction that is neither parallel nor orthogonal to any $L_i$.  This gives us in total $M/2$ small pairs of parallel hyperplanes.  We distribute each measure uniformly in one of these $M$ little segments.

    For any two directions $L_i, L_j$, if we project a pair of parallel segments onto $\Pi= \operatorname{span
    }(L_i, L_j)$, we obtain exactly the two-dimensional example constructed by Karasev, Rold\'an-Pensado, and Sober\'on \cite{Karasev:2016cn}.  It is impossible to bisect both projected segments using a line orthogonal to $L_i$ and a line orthogonal to $L_j$.

    Now we go back to $\rr^d$, and assume we have a valid hyperplane arrangement.  If any pair of segment is intersected by three or more hyperplanes, since we have $M/2$ pairs of segments and $M$ hyperplanes, there must be another pair of segments that is intersected at most one hyperplane.  Using just one hyperplane it is impossible to bisect the pair, and we would be done.  If every pair of segments is intersected by exactly two hyperplanes, by construction we can only bisect all of them if each pair is intersected by a pair of parallel hyperplanes.  Since there is an odd number of hyperplanes orthogonal to $L_1$, at least one of them will be paired with a hyperplane with another fixed direction, so there will be a pair of parallel segments we are not bisecting.

    \textbf{Case 2.} $M$ is odd.

    In this case, we construct $(M+1)/2$ pairs of parallel segments as before, and then remove one segment of the final pair.  Assume that we have a valid arrangement of hyperplanes that bisects this last segment.  If at least two hyperplanes cut this last segment, then we do not have enough hyperplanes to bisect the remaining $(M-1)/2$ pairs of tiny parallel segments.  If exactly one hyperplane cuts the last segment, we may assume without loss of generality that is is not one of the hyperplanes orthogonal to $L_1$ (otherwise, interchange the roles of $L_1$ and $L_2$ in the rest of the argument).  Then, since we still have to use an odd number of hyperplanes orthogonal to $L_1$, the argument follows as in Case $1$.
\end{proof}

\subsection{A  counterexample to Langerman's conjecture}

As mentioned in the introduction, Langerman's conjecture has been confirmed when $d$ is a power of two by Hubard and Karasev \cite{Hubard2020}. The general conjecture corresponds to the case $k_1 = \dots = k_n = 1$ and $L_1 = \dots = L_n = \rr^d$ of \cref{thm:main}.  Our conditions on the parity of $N$ match the conditions found by Hubard and Karasev, which cannot be improved in these dimensions.

\begin{theorem} There exists a set of six mesures in $\mathbb R^3$ that cannot be bisected by two planes.
\end{theorem}

The precise construction was built by trying to extend the Karasev--Rold\'an-Pensado--Sober\'on construction \cite{Karasev:2016cn} described in the previous subsection. 
The example is based on the idea of having three pairs of parallel segments, each concentrated near a point, with well chosen segment directions. We discretized this to a measure supported on three points almost aligned with the segment. We call each of these measures a \emph{needle}.

   Since needles are supported on three points each, the validity of the counterexample amounts to checking \[
3^6 \frac{\binom{6}{3}}{2}.
=7290\]
pairs of planes, each of which is pinned by three points of the needles, and also six points, one from each measure, lie exactly on the union of the planes. 

To find the precise points, an LLM was used. The first example we obtained was in generic position. The example we produce here is simpler. It is not in general position, but none of the relevant planes contains more than three points. \vspace{5mm}

\begin{tabular}{cccc}

$S_1:$ & $(1050,-410,-61)$ & $(1106,-408,-59)$ & $(1169,-410,-61)$ \\
$S_2:$ & $(1045,-409,-57)$ & $(1101,-407,-55)$ & $(1164,-409,-57)$ \\
$S_3:$ & $(2,-2,60)$       & $(3,-3,0)$        & $(2,-2,-60)$      \\
$S_4:$ & $(-3,3,61)$       & $(-2,2,0)$        & $(-3,3,-61)$      \\
$S_5:$ & $(-831,1000,35)$  & $(-830,1059,35)$  & $(-831,1119,35)$  \\
$S_6:$ & $(-833,1001,38)$  & $(-832,1060,38)$  & $(-831,1120,38)$  \\

\end{tabular}\\

Notice that,

\[ c_x = S_2 - S_1 = (-5,\,1,\,4),\] \[ c_y = S_6 - S_5 = (-2,\,1,\,3), \]
\[c_{z,j}= S_4 - S_3 = (-5,\,5,\,-j), \quad j \in \{-1,0,1\}. \]
\vspace{2mm}

The translation vectors $c_x$ and $c_y$ are not far from being aligned, and not far from being perpendicular to $c_{z,j}$. Measures $S_1$ and $S_2$ are almost parallel to the $x$-axis, we call these measures the $x$-cluster. Measures $S_5,S_6$ are almost parallel to the $y$-axis, we call these measures the $y$-cluster. Finally, measures
$S_3$ and $S_4$ are parallel to the $z$-axis. It seems that all constructions near this one of the form $S_4=S_3+v$ for some constant vector $v$ fail. Experimentally, $8$ bisecting arrangements appear when we tried a constant translation vector. 
To eliminate them the shear on the translation of from one needle on the $z$-cluster to the other is required.



We had a computer program test all 7290 possible candidates and none were bisecting all six sets.  The script is available at \href{https://github.com/chiral-polytope/two-plane-checkerboard-verifier}{https://github.com/chiral-polytope/two-plane-checkerboard-verifier}
\begin{figure}
\begin{tikzpicture}[
  ax/.style={-{Stealth[length=4pt]},line width=0.2pt,black!55},
  ndl/.style={line width=0.6pt},
  pt/.style={circle,fill,inner sep=0.9pt}]
  \xyz{Xm}{-1000}{0}{0}\xyz{Xp}{1300}{0}{0}
  \xyz{Ym}{0}{-550}{0}\xyz{Yp}{0}{1250}{0}
  \xyz{Zm}{0}{0}{-380}\xyz{Zp}{0}{0}{380}
  \draw[ax] (Xm)--(Xp); \draw[ax] (Ym)--(Yp); \draw[ax] (Zm)--(Zp);
  \xyz{A1}{1050}{-410}{-61}\xyz{A0}{1106}{-408}{-59}\xyz{A2}{1169}{-410}{-61}
  \draw[ndl,nA] (A1)--(A2); \node[pt,nA] at (A1){}; \node[pt,nA] at (A0){}; \node[pt,nA] at (A2){};
  \xyz{B1}{1045}{-409}{-57}\xyz{B0}{1101}{-407}{-55}\xyz{B2}{1164}{-409}{-57}
  \draw[ndl,nB] (B1)--(B2); \node[pt,nB] at (B1){}; \node[pt,nB] at (B0){}; \node[pt,nB] at (B2){};
  \xyz{C1}{2}{-2}{60}\xyz{C0}{3}{-3}{0}\xyz{C2}{2}{-2}{-60}
  \draw[ndl,nC] (C1)--(C2); \node[pt,nC] at (C1){}; \node[pt,nC] at (C0){}; \node[pt,nC] at (C2){};
  \xyz{D1}{-3}{3}{61}\xyz{D0}{-2}{2}{0}\xyz{D2}{-3}{3}{-61}
  \draw[ndl,nD] (D1)--(D2); \node[pt,nD] at (D1){}; \node[pt,nD] at (D0){}; \node[pt,nD] at (D2){};
  \xyz{E1}{-831}{1000}{35}\xyz{E0}{-830}{1059}{35}\xyz{E2}{-831}{1119}{35}
  \draw[ndl,nE] (E1)--(E2); \node[pt,nE] at (E1){}; \node[pt,nE] at (E0){}; \node[pt,nE] at (E2){};
  \xyz{F1}{-833}{1001}{38}\xyz{F0}{-832}{1060}{38}\xyz{F2}{-833}{1120}{38}
  \draw[ndl,nF] (F1)--(F2); \node[pt,nF] at (F1){}; \node[pt,nF] at (F0){}; \node[pt,nF] at (F2){};
\end{tikzpicture}
\end{figure}

\section{Acknowledgments}

This work started during the 2022 ``Extremal Combinatorics and Geometry'' workshop at BIRS, we thank both BIRS and the workshop organizers. We also thank Xavier Goaoc, Edgardo Rold\'an-Pensado, and Patrick Schnider for the helpful discussions.


\end{document}